\documentclass[12pt,article,reqno]{amsart}
\usepackage{amsmath, amsthm, amsfonts,mathrsfs} 
\usepackage[usenames,dvipsnames]{xcolor} 
\usepackage[top=29mm,bottom=29mm,left=30mm,right=30mm]{geometry} 

\usepackage[bookmarks, colorlinks, breaklinks, pdftitle={Hypercontractivity on the Poisson space},pdfauthor={XY}]{hyperref} 
\hypersetup{linkcolor=blue,citecolor=blue,filecolor=black,urlcolor=MidnightBlue} 

\makeatletter
\def\section{\@startsection{section}{1}%
\z@{1\linespacing\@plus\linespacing}{1\linespacing}%
{\bf\centering}}
\def\subsection{\@startsection{subsection}{0}%
\z@{\linespacing\@plus\linespacing}{\linespacing}%
{\bf}}
\def\subsubsection{\@startsection{subsubsection}{0}%
\z@{\linespacing\@plus\linespacing}{\linespacing}%
{\bf}}
\makeatother

\makeatletter
\@addtoreset{equation}{section}
\makeatother

\newtheorem{theorem}{Theorem}[section]
\newtheorem{corollary}[theorem]{Corollary}
\newtheorem{lemma}[theorem]{Lemma}

\newtheorem{fact}[theorem]{Fact}
\theoremstyle{definition}

\newtheorem{remark}[theorem]{Remark}

\def\RR{\mathbb{R}}\def\R{\mathbb{R}}

\def\NN{\mathbb{N}}

\def\PP{\mathbb{P}}
\def\EE{\mathbb{E}}
\def\Var{\mathbb{V}\mathrm{ar}}

\def\Ent{\mathrm{Ent}}

\def\ga{\gamma}
\def\de{\delta}

\def\ep{\varepsilon}

\def\la{\lambda}
\def\sk{\smallskip}

\def\ov{\overline}

\def\wh{\widehat}

\newcommand{\norm}[1]{\left\lVert#1\right\rVert}

\newcommand{\ind}[1]{1_{\{#1\}}}
\newcommand{\red}[1]{{\color{black}{#1}}} 

\newcommand{\blue}{\textcolor[rgb]{1.00,1.00,1.00}}
\renewcommand{\qed}{\hfill$\square$}

\DeclareMathOperator\dom{dom} 

\begin{document}
\title
{Restricted hypercontractivity on the Poisson space}
\author{Ivan Nourdin \and Giovanni Peccati \and Xiaochuan Yang}

\address{Ivan Nourdin, Universit\'e de Luxembourg, Maison du Nombre, 6 avenue de la Fonte, L-4364 Esch-sur-Alzette, Grand
Duchy of Luxembourg.}
\email{ivan.nourdin@uni.lu}

\address{Giovanni Peccati, Universit\'e de Luxembourg, Maison du Nombre, 6 avenue de la Fonte, L-4364 Esch-sur-Alzette, Grand
Duchy of Luxembourg.}
\email{giovanni.peccati@uni.lu}

\address{Xiaochuan Yang, Universit\'e de Luxembourg, Maison du Nombre, 6 avenue de la Fonte, L-4364 Esch-sur-Alzette, Grand
Duchy of Luxembourg.}
\email{xiaochuan.yang@uni.lu}

\thanks{\emph{Key-words}: hypercontractivity, improved Poincar\'e inequality, modified logarithmic Sobolev inequality, Poisson processes, Emery-Bacry criteria, superconcentration.
 \\ \medskip
\noindent
2010 {\it MS Classification}: 60H07; 60E15; 60E05  \\
}

\begin{abstract} We show that the Ornstein-Uhlenbeck semigroup associated with a general Poisson random measure is hypercontractive, whenever it is restricted to non-increasing mappings on configuration spaces. We deduce from this result some versions of Talagrand's $L^1$-$L^2$ inequality for increasing and concave mappings, and we build examples showing that such an estimate represents a strict improvement of the classical Poincar\'e inequality. We complement our finding with several results of independent interest, such as gradient estimates. 

\end{abstract}

\maketitle


\baselineskip 0.5 cm

\bigskip \medskip

\section{Introduction} 

The aim of this paper is to study some form of {\it restricted hypercontractvity} for the Ornstein-Uhlenbeck semigroup associated with a Poisson random measure, with specific focus on the derivation of variance estimates in the spirit of Talagrand's $L^1$-$L^2$ inequality \cite{CEL, T}. This complements several findings about concentration and logarithmic Sobolev inequalities on configuration spaces, whose interest has recently been revived by geometric applications: see \cite{B, BP, BR, GLO, PR, R} for results in a geometric context, as well as the classical references \cite{AL, BL, BT, BHP, HP, W} for theoretical foundations. 

We start by recalling some classical facts about Gaussian measures, to which our findings should be compared. 

\sk

For the rest of the paper, generic random objects are defined on an appropriate common probability space $(\Omega, \mathcal{F}, \mathbb{P})$, with $\mathbb{E}$ and $\Var$ denoting, respectively, expectation and variance with respect to $\mathbb{P}$. When dealing with a specific probability measure $\mu$ on a measurable space $(A, \mathcal{A})$ and with a mapping $f : A\to \mathbb{R}$, we write 
$\mathbb{E}_\mu[ f] := \int_A f d\mu$, whenever this expression is well-defined.

\sk
\sk

\subsection{Some Gaussian estimates}

Standard references for the results discussed in this subsection are e.g. \cite{LS2000, BGL, L2000, NPbook}. For $k\geq 1$, let $\gamma_k$ denote the standard Gaussian measure on $\R^k$, that is: $\gamma_k(dz) := (2\pi )^{-k/2} e^{-\| x\|^2/2}dx$, where $\| \cdot\|$ stands for the Euclidean norm, and $dx = dx_1\cdots dx_k$. The {\it Poincar\'e inequality} states that, for every smooth $f\in L^2(\ga_k)$, 
\begin{align}\label{e:poin}
\Var_{\ga_k} (f):= \int f^2 d\ga_k - \left(\int f d\ga_k\right)^2 \le \sum_{i=1}^k \EE_{\ga_k} [(\partial_i f )^2],
\end{align}
where $\partial_i := \partial/{ \partial x_i}$, and the {\it logarithmic Sobolev inequality} states that, for $f$ smooth,
\begin{align}\label{e:logsob}
\Ent_{\ga_k} (f^2):=\int f^2\log f^2 d\ga_k - \int f^2d\ga_k\log\left( \int f^2 d\ga_k\right)\le 2 \sum_{i=1}^k \EE_{\ga_k} [(\partial_i f )^2],
\end{align}
which, by a change of variable and the chain rule for derivatives, is equivalent to the estimate
\begin{align}\label{e:LS_L1}
\Ent_{\ga_k}(f)\le \frac12 \sum_{i=1}^k \EE_{\ga_k}\left[ \frac{(\partial_i f)^2}{f}\right],
\end{align}
for every $f$ positive and smooth. Note that $\ga_k$ is the invariant distribution of the {\it Ornstein-Uhlenbeck semigroup}, which (in its Mehler's form) is defined as 
\begin{align*}
P^{\ga_k}_t f(x)= \int f(e^{-t}x+\sqrt{1-e^{-2t}}y)\ga_k(dy), \quad f\in L^1(\ga_k), \quad t\geq 0.
\end{align*}

The logarithmic Sobolev inequality \eqref{e:logsob} (or \eqref{e:LS_L1}) is actually equivalent to the hypercontractivity of $\{P_t^{\ga_k}\}$ \cite[p. 246]{BGL}, and the latter property is the key to prove Talagrand's $L^1$-$L^2$ bound for the variance \cite{CEL, T}. Recall that $\{P^{\ga_k}_t\}$ is {\it hypercontractive} in the following sense: for $p>1$ and $t\ge 0$,
\begin{align}\label{e:hyper_G}
\norm{P^{\ga_k}_t f}_{1+e^{2t}(p-1)} \le \norm{f}_p
\end{align}
where $\norm{g}^q_q = \EE_{\ga_k}[|g|^q]$ for any $q>1$, and Talagrand's $L^1$-$L^2$ bound on the Gaussian space \cite[Th.~1, $\kappa=\rho=1$]{CEL} states that, for each $f\in L^2({\ga_k})$ smooth, 
\begin{align} \label{e:tl1l2}
\Var_{\ga_k}(f) \le 4e^2 \sum_{i=1}^k  \frac{\norm{\partial_i f}^2_2}{1+\log(\norm{\partial_i f}_2/\norm{\partial_i f}_1)}. 
\end{align}

\smallskip

A bound analogous to \eqref{e:tl1l2} was originally proved in \cite{T} for functions defined on hypercubes, and later generalized in \cite{CEL} to the framework of Markov semigroups enjoying a form of hypercontractivity, yielding \eqref{e:tl1l2} as a special case. 
It is clear that Talagrand's bound strictly improves upon the Poincar\'e inequality \eqref{e:poin}, and that such an improvement becomes substantial whenever the $\partial_i f$'s are somehow close to an indicator function. For a self-contained proof of this inequality, and for applications to the superconcentration phenomenon, see \cite[Ch.5]{C}.  

\smallskip

It is a standard exercise to suitably adapt the approach of \cite{CEL, T} in order to lift \eqref{e:tl1l2} to an infinite-dimensional setting (we leave the details of the proof to the reader). To see this, let $(\mathbb X, \mathcal X)$ be a measurable space, let $\lambda$ be a $\sigma$-finite measure on it, and consider a {\it Gaussian measure}  $G$ with intensity $\lambda$ (see \cite[p. 24]{NPbook} for a definition). Then, the Ornstein-Uhlenbeck semigroup associated with $G$ is hypercontractive (see \cite[Section 2.8.3]{NPbook}). Moreover, given a square-integrable functional $F \in \sigma(G)$ in the domain of the Malliavin derivative $D$ associated with $G$ (see \cite[Section 2.3]{NPbook}), one can prove that
\begin{align} \label{e:gl1l2}
\Var(F)  \le 4e^2 \int_{\mathbb X}  \frac{\norm{D_x F}^2_2}{1+\log(\norm{D_x F}_2/\norm{D_x F}_1)} \lambda(dx), 
\end{align}
where in this case $\|D_x F\|_p^p := \mathbb{E}[| D_x F|^p] $. Relation \eqref{e:gl1l2} is a strict improvement of the classical Poincar\'e inequality for infinite-dimensional Gaussian fields, see e.g. \cite[Exercise 2.11.1]{NPbook}, and contains \eqref{e:tl1l2} as a special case.

\smallskip

One of the principal achievements of the present paper is the derivation of estimates analogous to \eqref{e:gl1l2} in the framework of the non-hypercontractive Ornstein-Uhlenbeck semigroup associated with a Poisson random measure, under some special assumptions on the random variable $F$.

\subsection{Basic inequalities for Poisson measures}

Let $(\mathbb X, \mathcal X, \lambda)$ be a $\sigma$-finite measure space. In what follows, we denote by $\eta = \{\eta(A) : A \in \mathcal{X}\}$ a {\it Poisson random measure} with intensity $\lambda$. Here, we observe that $\eta$ is a random element taking values in the space of countably supported, integer-valued mesures on $(\mathbb X, \mathcal X)$, and refer the reader to Section \ref{s:calculsto} for precise definitions. Poisson random measures are one of the most fundamental objects of modern probability theory, emerging in a number of theoretical and applied domains such as L\'evy processes, Brownian excursion theory, stochastic geometry, extreme values and queueing theory  -- see e.g. \cite{LP, PR} for an overview. Note that, in the case where $\mathbb {X} = \{x\}$ is a singleton, then $\eta$ can be identified with a one-dimensional Poisson random variable with parameter $\alpha := \lambda( \{x\})$. From now on, we write $\|F\|_p :=\mathbb{E}[|F|^p]^{1/p}$, $p\geq 1$. 

\smallskip

For every $x\in \mathbb{X}$, we denote by $D_x$ the {\it add-one cost operator} at $x$, defined as follows:
for every $F\in \sigma(\eta)$, $D_x F(\eta) := F(\eta+\delta_x) - F(\eta)$, where $\delta_x$ stands for the Dirac mass centered at $x$. We also write $\{P_t \}$ to indicate the Ornstein-Uhlenbeck semigroup associated with $\eta$, which is formally defined in formula \eqref{e:mehler} below (once again in its Mehler's form). We start by recalling the classical {\it Poincar\'e inequality} on the Poisson space (see e.g. \cite[p.~193]{LP}).
\begin{fact}[Poincar\'e inequality]\label{f:poincare}
Suppose $F\in L^2(\mathbb{P})$. Then
\begin{align}\label{e:poisspoin}
\Var(F) \le \int \EE[|D_x F|^2] \la(dx).
\end{align}
\end{fact}

In a Poisson setting, the estimate \eqref{e:poisspoin} plays a role analogous to \eqref{e:poin} (and of its infinite-dimensional counterpart) on a Gaussian space. On the other hand, as observed e.g. in the classical references \cite{BL, Surgailis}, the Ornstein-Uhlenbeck semigroup $\{P_t\}$ associated with $\eta$ is {\it not} hypercontractive, and a logarithmic Sobolev inequality analogous to \eqref{e:logsob} {\it cannot} hold, even in the simple case in which $\mathbb X$ is a singleton (see the discussion contained in Section \ref{sss:onepoiss} below). The next result is a {\it modified logarithmic Sobolev inequality} proved by Wu \cite[Th. 1.1]{W}, see also \cite{AL,BL, Chafai}, as well as \cite[p. 212]{PR}.

\begin{fact}[{Modified} logarithmic Sobolev inequality] Let $F\in L^1(\mathbb{P})$ be such that $F>0$ a.s.. Then, 
\begin{align}\label{e:wuls}
\mathrm{Ent}(F) := \mathbb{E}(\Phi(F) ) - \Phi(\mathbb{E}(F))  \le \EE\int [D_x\Phi(F)-\Phi'(F)D_xF] \la(dx),
\end{align}
where $\Phi(u)=u\log u$. 
\end{fact}

{For the rest of the paper we adopt the usual convention $\Phi(0) = 0\log 0 := 0$}. The bound \eqref{e:wuls} immediately yields the next statement, containing in particular an estimate analogous to \eqref{e:LS_L1}.

\begin{corollary}[See {\cite[Cor. 2.1 and 2.2]{W}}] Let $F\in L^1(\mathbb{P})$ be such that $F>0$ a.s.. Then,
\begin{align}\label{e:logsob^2}
\mathrm{Ent}(F)\le \EE\int \min\left(\frac{|D_xF|^2}{F}, D_xF D_x\log F\right) \la(dx).
\end{align}
\end{corollary}

\sk

A direct application of \cite[Theorem 6.1]{BT} implies that, as a consequence of \eqref{e:logsob^2}, the following form of {\it weak hypercontractivity} holds: for every bounded $F\in \sigma(\eta)$ and every $t\ge 0$, 
\begin{align}\label{e:whc}
\|e^{P_t F} \|_{e^t}\leq \|e^F\|_1.
\end{align}
To the best of our expertise, it does not seem possible to use \eqref{e:whc} in order to deduce any meaningful extension of \eqref{e:gl1l2} to generic functionals of $\eta$.


\subsection{Main results}\label{ss:main}

The goal of this note is twofold. On the one hand, we will determine a subset of functionals of a general Poisson point process on which the Ornstein-Uhlenbeck semigroup enjoys a hypercontractivity property analogous to \eqref{e:hyper_G}; we call such a property {\it restricted hypercontractivity}. On the other hand, we will apply the restricted hypercontractivity of $\{P_t\}$ in order to obtain several $L^1$-$L^2$ bounds in the spirit of \eqref{e:gl1l2}.

\sk

Our first result is the following. 


\begin{theorem}[Restricted hypercontractivity]\label{t:hyper}
Let $\eta$ be a Poisson point process on a measurable space $(\mathbb X, \mathcal X)$ with $\sigma$-finite intensity measure $\la$. Let $F\ge 0$ be $\sigma(\eta)$-measurable and such that $D_x F \le 0$ for all $x\in\mathbb{X}$. Then, for $t\ge 0$ and $p>1$, 
\begin{align}\label{e:hyper_P}
\norm{P_t F}_{1+(p-1)e^t}\le \norm{F}_p,
\end{align}
where $\{P_t\}$ is the Ornstein-Uhlenbeck semigroup associated with $\eta$.
\end{theorem}

\begin{remark}
\begin{itemize}
\item[(a)] One difference between the restricted hypercontractivity of the Ornstein-Uhlenbeck semigroup in the Poisson setting and that in the Gaussian setting \eqref{e:hyper_G} is the factor $e^{t}$ replacing $e^{2t}$ in $\eqref{e:hyper_P}$. Such a factor is consistent with the constants appearing in the logarithmic Sobolev inequality for decreasing functionals, obtained from the forthcoming Corollary \ref{c:G^q} by setting $q=2$ (that should be compared with \eqref{e:logsob}). See e.g. \cite[Section 2.8 and Section 5.4]{LS2000}.

\item [(b)] The computations leading to the proof of Theorem \ref{t:hyper} (in particular, Corollary \ref{c:G^q} combined with a standard implementation of Herbst's argument) also implicitly yield the following concentration estimate from \cite[Proposition 3.1]{W}: if $F \in \sigma(\eta)$ is such that $D_xF\leq 0$ for every $x\in \mathbb{X}$, and $\int_\mathbb{X} (D_xF )^2\lambda(dx)\leq \alpha^2<\infty$, then $F$ is integrable and, for $t>0$,
$$
\mathbb{P}[ F - \mathbb{E}(F)>t]\leq \exp\left\{ -\frac{t^2}{2\alpha^2}\right\}.
$$
See the forthcoming Section \ref{ss:entroproofs}.

\end{itemize}
\end{remark}

We now apply Theorem \ref{t:hyper} to either $D_xF$ or $-D_xF$ for each $x$ and obtain the following $L^1$-$L^2$ bound. For the rest of the paper, and for every $x,y\in \mathbb{X}$, we write $D_{x,y}^{2} = D_xD_y = D_yD_x$, in such a way that
$$
D_{x,y}^{2}F(\eta) = F(\eta+\delta_x+\delta_y) - F(\eta+\delta_x)-F(\eta+\delta_y)+F(\eta).
$$

\begin{theorem}[Talagrand's $L^1$-$L^2$ bound]\label{t:tal} Let $F\in \sigma(\eta)$ be a square-integrable functional satisfying either {\rm (i)} $D_xF\ge  0$ and $D_{x,y}^{2}F\le  0$ for every $x,y\in \mathbb{X}$, or {\rm (ii)} $D_xF\le  0$ and $D_{x,y}^{2}F\ge  0$ for every $x,y\in \mathbb{X}$.  Then,
\begin{align}\label{e:pl1l2}
\Var(F) \le \frac12 \int \frac{\norm{D_xF}_2^2}{1+\log(\norm{D_xF}_2/ \norm{D_xF}_1)} \la(dx).
\end{align}
\end{theorem}

Requirements (i) and (ii) in the statement of Theorem \ref{t:tal} have to be interpreted in the following way: for (i), we require that $D_xF \geq 0$ a.e.-$\mathbb{P}\otimes \lambda$ and that $D_{x,y}^{2}F\le  0$, a.e.-$\mathbb{P}\otimes \lambda^2$, and the same relations with reversed inequalities for (ii). The discussion contained in Section \ref{ss:examples} shows that (unlike in the Gaussian setting) a bound such as \eqref{e:pl1l2} cannot hold in full generality, even in the case of a one-point space $\mathbb{X} = \{x\}$. 

\sk

A similar method based on restricted hypercontractivity can be applied to prove a $L^1$ bound for the variance. This is a Poisson counterpart to \cite[Th.~6]{CEL}.

\begin{theorem}\label{t:L^1} Let $F$ be a bounded functional. Suppose  either (i) $D_x F\ge  0$ and $D^{2}_{x,y}F\le  0$ for every $x,y\in \mathbb{X}$, or (ii) $D_x F\le  0$ and $D^{2}_{x,y}F\ge  0$ for every $x,y\in \mathbb{X}$. Then,  with $\alpha(F)=1$ if $2\norm{F}_\infty>1$ and $\alpha(F)=2/(e+1)$ otherwise,

\begin{align*}
\Var(F)\leq
11 \, (2\norm{F}_\infty)^{\alpha(F)}
\int \la(dx)\times \left\{
\begin{array}{cl}
2\int \frac{\la(dx)}{1+\log(1/\EE[|D_xF|])}&\mbox{if $\EE[|D_xF|]\leq 1$}\\
\EE[|D_xF|]&\mbox{if $\EE[|D_xF|]\geq 1$}\\
\end{array}
\right..
\end{align*}
\end{theorem}

\begin{remark}
{The previous} bound in the case where $\EE[|D_xF|]\leq 1$ differs from that of \cite[Th.~6]{CEL} by a square root of the denominator. This is due to our use of the commutativity property \eqref{e:commutativity}, which yields a pointwise gradient estimate that is very different from the Gaussian case -- compare \cite[Equation (25)]{CEL} with our \eqref{e:grad_P}. In Section \ref{s:grad}, we will show an integrated gradient estimate closer to the Gaussian case, which is of independent interest -- see Theorem \ref{t:grad_integ}.     
\end{remark}

\subsection{Two examples}\label{ss:examples}
 
 \subsubsection{The one-dimensional Poisson distribution}\label{sss:onepoiss} Let $X_\lambda$ be a Poisson random variable with parameter $\lambda$. For every function $G:\NN\to\RR_+$ and every $n\in \NN$, set $DG(n) = G(n+1) - G(n)$, and $D^2G(n) = DDG(n)$. In this simple framework, the Poincar\'e inequality \eqref{e:poisspoin} boils down to the statement: if $ G(X_\lambda)$ is square-integrable, then $\Var  \, G(X_\lambda) \le \lambda \mathbb{E}[ (DG(X_\lambda))^2]$. Similarly, if $DG(n)\geq 0$ and $D^2G(n)\leq 0$, then one deduces from \eqref{e:pl1l2} that
\begin{equation}\label{e:simplet}
 \Var  \, G(X_\lambda) \le 2 \lambda \frac{ \EE[ (DG(X_\lambda)^2] }{1+\log(\norm{DG(X_\lambda) }_2/ \norm{DG(X_\lambda)}_1)}.
\end{equation}
To apply and compare the previous estimates, let $g:\NN\to\RR_+$ be positive and non-increasing. Then, the function $G: \NN\to\RR_+$ defined by $G(0)=0$ and $G(n)=\sum_{j=0}^{n-1}g(j)$, $n\ge 1$, satisfies, for $n\geq 0$,
\begin{align*}
DG(n)=G(n+1)-G(n)=g(n), \quad \mbox{and} \quad D^2G(n)=Dg(n)\le 0. 
\end{align*}
By the Poincar\'e inequality and the $L^1$-$L^2$ inequality stated above, one therefore deduces that
\begin{align*}
\Var \,  G(X_\la)\le \lambda \EE[g(X_\la)^2]
\end{align*}
and 
\begin{align*}
\Var \, G(X_\la) \le \frac{2\lambda \EE[g(X_\la)^2]}{1+\log \frac{\norm{g(X_\la)}_2}{\norm{g(X_\la)}_1}}.
\end{align*}
We can easily devise examples where the $L^1$-$L^2$ bound dominates the Poincar\'e bound, as $\la\uparrow\infty$. For instance, one can take $g(j)=\ind{j\le M}$ for some $M\ge 1$. Then \begin{align*}
\norm{g(X_\la)}_2 = \sqrt{\norm{g(X_\la)}_1}
\end{align*}
and \begin{align*}
\norm{g(X_\la)}_1 = \PP(X_\la\le M) = e^{-\la} \sum_{k=0}^M \frac{\la^k}{k!} \to 0, 
\end{align*}
so that
\begin{align*}
\log \frac{\norm{g(X_\la)}_2}{\norm{g(X_\la)}_1} = -\frac12\log\norm{g(X_\la)}_1\to \infty. 
\end{align*}
It is a fundamental observation that the estimate \eqref{e:simplet} cannot hold for general random variables $G(X_\lambda)$. To see this, let $X = X_1$ be a Poisson random variable with mean 1, and, for some integer $k\geq 2$, define the random variable $F_k = 1_{\{X\leq  k-1\}}$, in such a way that $DF_k(X) = -1_{\{X=  k-1\}}$. Then, as $k\to \infty$,
$$
\Var(F_k) \sim \frac{e^{-1}}{k!}, \quad \mbox{and}\quad \mathbb{E}[DF_k(X)] =\mathbb{E}[DF_k(X)^2] =  \frac{e^{-1}}{(k-1)!};
$$
however, an application of Stirling's formula yields
$$
1+\log(\norm{DF_k}_2/ \norm{DF_k}_1) \sim \frac12  k\log(k),
$$
thus proving that \eqref{e:simplet} (and therefore \eqref{e:pl1l2}) cannot hold in general. It is interesting to notice that the random variable $1-F_k = 1_{\{X\geq  k\}}=: f_k(X)$ can be used in order to show that an inequality analogous to \eqref{e:logsob} cannot hold on a Poisson space. To see this, denote by $\pi$ the law of $X_1$, and assume that there exists a finite positive constant $C$ such that, for $f$ bounded on $\NN$,
\begin{align}\label{e:LS_fake}
\Ent_\pi(f^2) \le C \EE_\pi[|Df|^2].
\end{align}
Applying such an estimate to $f = f_{k+1}$, one infers that
\begin{align*}
-\pi([k+1,\infty))\log \pi([k+1,\infty))\le C \pi(k)
\end{align*} 
which is seen to be absurd, by letting $k\to\infty$. See \cite{BL} for a full discussion.

 \subsubsection{{Maxima} }
%

Let $\mu$ be a probability measure on $(\RR^d,\mathcal B(\RR^d))$, $d\geq 1$.  We denote by $\eta_n$ a Poisson measure on $\RR^d$, with intensity $\la_n(dx)=n\mu(dx)$.  Write $\norm{x}$ for the Euclidean norm of $x=(x_1,...,x_d)\in\RR^d$.   Assume that $\mu$ is diffuse so that each $\eta_n$ is simple, that is, every point in the support of $\eta_n$ is charged with mass 1 (see Section \ref{s:calculsto}); in particular, with a slight abuse of notation, in what follows we will identify $\eta_n$ and its support.  For every $t>0$, define 
\begin{align*}
F:=F(t,n)=1_{\{\max_{x\in\eta_n}\norm{x}>t\}}.
\end{align*}
One has 
\begin{align*}
\Var F= \PP(\max_{x\in\eta_n}\norm{x}>t)\PP(\max_{x\in\eta_n}\norm{x}\le t).
\end{align*}
Now writing $B(0,t)=\{y:\norm{y}\le t\}$, we have
\begin{align*}
\PP(\max_{x\in\eta_n}\norm{x}\le t) = \PP(\eta_n(B(0,t)^c)=0) =e^{-n\mu(B(0,t)^c)}.
\end{align*}
One has also that for $z\in\RR^d$,
\begin{align*}
D_z F&= \ind{\max_{x\in\eta_n+\de_z}\norm{x}>t} - \ind{\max_{x\in\eta_n}\norm{x}>t}\\
&=\begin{cases}
1 & \norm{z}>t \mbox{ and } \max_{x\in\eta_n}\norm{x}\le t, \\
0 & \mbox{ otherwise. }
\end{cases}
\end{align*}
Consequently,
\begin{align*}
\EE \int (D_zF)^2\la_n(dz) = n\int_{B(0,t)^c} e^{-n\mu(B(0,t)^c)} \mu(dz)= ne^{-n\mu(B(0,t)^c)}\mu(B(0,t)^c),
\end{align*}
implying that, in this framework, the Poincar\'e inequality is suboptimal by a factor of $n$. We argue that the $L^1$-$L^2$ inequality provides the right order for the variance.  For this, note first that 
\begin{align*}
D_z F(\eta_n+\de_y) = \begin{cases}
1 & \norm{z}>t \mbox{ and }
\max(\max_{x\in\eta_n} \norm{x}, \norm{y})\le t, \\
0 & \mbox{ otherwise,}
\end{cases}
\end{align*}
yielding $D^{2}_{y,z} F \le 0$. 
Therefore, 
\begin{align*}
\norm{D_z F}_{L^2(\PP)}= \sqrt{e^{-n\mu(B(0,t)^c)}}\ind{\norm{z}>t} =\sqrt{\norm{D_z F}_{L^1(\PP)}} 
\end{align*}
and for $z$ with $\norm{z}>t$, 
\begin{align*}
\log \frac{\norm{D_z F}_{L^2(\PP)}}{\norm{D_z F}_{L^1(\PP)}} = \frac{n}{2}\mu(B(0,t)^c).
\end{align*}
Theorem \ref{t:tal} yields that 
\begin{align*}
\Var F(t,n)&\le 2 \frac{n e^{-n\mu(B(0,t)^c)}\mu(B(0,t)^c)}{1+\frac{n}{2}\mu(B(0,t)^c)}\\
&\sim 4 e^{-n\mu(B(0,t)^c)}
\end{align*}
as $n\to \infty$, as desired.

%
%

\subsection{Acknowledgments}

Earlier versions of Theorem \ref{t:hyper} and Theorem \ref{t:tal} {above} were communicated to G. Peccati by Sascha Bachmann in January 2016. Sascha declined our offer to become a co-author of the present paper and generously allowed us to use his ideas and results: we heartily thank him for this. We are grateful to M. Ledoux for useful discussions. I. Nourdin is supported by the FNR grant APOGee at Luxembourg University; G. Peccati is supported by the FNR grant FoRGES (R-AGR-3376-10) at Luxembourg University; X. Yang is supported by the FNR Grant MISSILe (R-AGR-3410-12-Z) at Luxembourg and Singapore Universities.

\section{Stochastic analysis for Poisson measures: basic definitions and results } \label{s:calculsto}

We adopt the notation and follow the presentation of \cite{L,LP}.  Let $(\mathbb X, \mathcal X)$ be a measurable space, and write $\NN_0=\NN\cup\{0\}$ and $\ov\NN_0 = \NN_0\cup\{\infty\}$. Let $\mathbf N_{<\infty}$ be the space of $\NN_0$-valued measures on $\mathcal X$.  Define $\mathbf N$ as the space of measures which can be written as a countable sum of elements in $\mathbf N_{<\infty}$.  Equip $\mathbf N$ with the smallest $\sigma$-algebra $\mathcal N$ generated by the sets $\{\mu\in \mathbf N: \mu(B)=k \}$, for all $B\in\mathcal X$ and $k\in\NN_0$.  

\sk

Let $(\Omega,\mathcal F, \PP)$ be a probability space. A point process $\eta$ is a measurable map from $(\Omega, \mathcal F)$ to $(\mathbf N, \mathcal N)$.  A Poisson random measure $\eta$ is a point process that satisfies the following properties: (i) for any $B\in\mathcal X$, the random variable $\eta(B)$ is Poisson distributed; (ii) for any $m\in\NN$ and disjoint sets $B_1,...,B_m\in\mathcal X$, the family of random variables $(\eta(B_i))_{1\le i\le m}$ are mutually independent.  The measure $\la$ defined by $\la(\cdot)=\EE[\eta(\cdot)]$ is called the {\it intensity} of the Poisson measure $\eta$.   By virtue of \cite[Cor. 3.7]{LP}, up to equality in distribution, every Poisson process with $\sigma$-finite intensity is {\it proper}, in the sense that there exists a sequence of $\mathbb X$-valued random variables $X_i$ and a $\ov\NN_0$-valued random variable $\kappa$ such that 
\begin{align}\label{e:proper}
\eta = \sum_{n=1}^\kappa \de_{X_n}.
\end{align}
Also, we have the classical Mecke's formula (see \cite[Ch. 4]{LP}), valid for all measurable $h:\mathbf N\times \mathbb{X}\to [0,\infty]$:
\begin{equation}\label{mecke}
\EE\int h(\eta,x)\eta(dx) = \EE\int h(\eta+\delta_x,x)\lambda(dx).
\end{equation}

\sk

In this note, we study square-integrable functionals on the canonical space $(\mathbf N, \mathcal N)$ of a Poisson {point} process $\eta$ with $\sigma$-finite intensity $\la$.  The most basic operation on $F\in L^0(\PP)$ is the {\it add-one-cost operator}, defined by $D_x F(\eta)= F(\eta+\de_x)-F(\eta)$. We then define recursively $D^{n}F=D(D^{n-1} F)$.   Each $F\in L^2(\PP)$ admits a Wiener-It\^o chaos expansion \cite[p.195]{LP}
\begin{align}\label{e:chaos}
F = \sum_{q\ge 0} I_q(f_q)
\end{align}
where the series converges in $L^2(\PP)$, $I_q(f_q)$ is the $q$-th multiple integral of $f_q$ with respect to the compensated Poisson process $\wh\eta=\eta-\la$ and $f_q: \mathbb X^q\to \RR$ is a $\la^q$-a.e. symmetric function whose explicit form is given by $f_q(x_1,...,x_q)=\frac{1}{q!}\EE[D^{q}_{x_1,...,x_n}F]$. 

\sk

Let $F\in L^2(\PP)$ have the chaos expansion \eqref{e:chaos}. Denote by $\dom D$ the set of $F\in L^2(\PP)$ satisfying $$\sum_{q\ge 1}q q! \norm{f}_{L^2(\la^q)}^2<\infty. $$ 
  By \cite[Th. 3]{L}, $F\in\dom D$ if and only if $DF\in L^2(\PP\otimes\la)$. In this case we have $\PP$-a.s. and for $\la$-a.e. $x\in\mathbb X$ that
\begin{align*}
D_x F= \sum_{q\ge 1} q I_{q-1}(f_q(x,\cdot)).
\end{align*}
Hence, the add-one-cost operator coincides with the Malliavin derivative for Wiener-Itô multiple integrals.  Let $\de$ be the adjoint operator of $D$, that is, for $H$ in the domain of $\de$, we have
\begin{align*}
\EE[ \langle H, DF\rangle_{L^2(\la)}] = \EE[\de(H)F] \quad \mbox{ for all } F\in\dom D.
\end{align*}
Finally, let  $\dom L$ be the subclass class of $F\in L^2(\PP)$ such that 
\begin{align*}
\sum_{q\ge 1} q^2 q! \norm{f_q}^2_{L^2(\la^q)}<\infty.
\end{align*} 
In this case we define
\begin{align*}
LF:= - \sum_{q\ge 1} qI_q(f_q).
\end{align*}
The (pseudo)-inverse $L^{-1}$ of $L$ is given by 
\begin{align*}
L^{-1}F:=-\sum_{q\ge 1}\frac{1}{q}I_q(f_q).
\end{align*}
The operator $L$ is the generator of the Orstein-Uhlenbeck semigroup $(P_t)_{t\ge 0}$, yet to be recalled, which interpolates between $F$ and its expectation. Recall that $\eta$ is proper as in \eqref{e:proper}.  Let $(U_i)_{i\ge 1}$ be a sequence of iid uniform random variables in $[0,1]$ and independent of $(\kappa,(X_n)_{n\ge 1})$. Define 
\begin{align*}
\eta_u:=\sum_{n=1}^{\kappa}\ind{U_n\le u} \de_{X_n}, \quad  0\le u\le 1.
\end{align*} 
Then $\eta_u$ is a $u$-thinning of $\eta$, see \cite[Ch. 5]{LP} for definition. Define for $F\in L^1(\PP)$
\begin{align}\label{e:mehler}
P_t F = \EE[F(\eta_{e^{-t}} + \eta'_{1-e^{-t}})|\eta], \quad t\ge 0,
\end{align}
where $\eta'_{1-e^{-t}}$ is a Poisson process with intensity $(1-e^{-t})\la$, independent of the pair $(\eta, \eta_{e^{-t}})$. Note that $P_0 F=F$, $P_\infty F=\EE F$, and $$\EE[P_tF]=\EE F, \quad F\in L^1(\PP).$$ Moreover,  \eqref{e:mehler} yields a commutation relation between $D$ and $P_t$, {see \cite[p.212]{LP}}
\begin{align}\label{e:commutativity}
D_x (P_t F) = e^{-t} P_t D_x F,  \quad \la\mbox{- a.e.} \,x\in\mathbb X, \PP\mbox{- a.s.}.
\end{align}
In particular, when $|F|\le 1$, we have the pointwise gradient estimate
\begin{equation}\label{e:grad_P}
|D_x (P_t F)|\le 2e^{-t}, \quad {\la\otimes \PP - a.e.}
\end{equation}
For $F\in L^2(\PP)$ with the chaos expansion \eqref{e:chaos}, we have by \cite[p.27]{L} 
\begin{align*}
P_t F = \sum_{q\ge 0} e^{-qt} I_q(f_q), \quad t\ge 0. 
\end{align*}

For $F,G\in\dom L$ such that $FG\in\dom L$, the {\it carr\'e-du-champs operator} is defined by $\Gamma(F,G)=\frac{1}{2}(L(FG)-FLG -GLF)$. Since $L$ is {\it symmetric} (meaning that $\EE[FLG]=\EE[GLF]$) we have
\begin{align*}
\EE[GLF]=\EE[FLG]=-\EE[\Gamma(F,G)].
\end{align*}
Using the pathwise representation \cite[p.1888]{DP} of $\Gamma$ and Mecke's formula (\ref{mecke}), we also obtain that, under suitable integrability assumptions on $DF, DG$,
$$\EE[\Gamma(F,G)]=\int \EE[D_x F D_x G]\la(dx).$$

\section{Proofs}\label{s:proof}

\subsection{Entropy estimates}\label{ss:entroproofs}
The main result of this subsection is Corollary \ref{c:G^q}. The following pathwise inequality will be useful. 

\begin{lemma}\label{l:tech}
{Let $G$ be $\sigma(\eta)$-measurable}. Then, for any $q>1$, $x\in\mathbb{X}$ and $\eta\in\mathbf{N}$, we have 
\begin{align*}
\frac{(D_x|G|^q)^2}{|G|^q}\le \frac{q^2}{q-1}(D_x|G|^{q-1})(D_x|G|)\left( \left|\frac{G(\eta+\de_x)}{G}\right|^q \vee 1\right),
\end{align*}
{where we have implicitily adopted the convention $1/0 = +\infty$.}
\end{lemma}
\begin{proof}
Set $a=|G(\eta+\de_x)|$ and $b=|G(\eta)|$. The inequality holds trivially if $a=b$ or $b=0$. Assume that $a\neq b>0$.  We intend to prove  
\begin{align*}
\frac{(a^q-b^q)^2}{b^q}\le \frac{q^2}{q-1} (a-b)(a^{q-1}-b^{q-1})\left(\left(\frac{a}{b}\right)^q\vee 1 \right).
\end{align*}
By the variable change $(a,b)\mapsto (\frac{a}{b},1)$, we can assume $b=1$ and this amounts to 
\begin{align*}
(a^q-1)^2  \le \frac{q^2}{q-1} {(a-1)(a^{q-1}-1)} (a^q \vee 1).
\end{align*}
We have
\begin{align}\label{e:aq1}
 |a^q-1| = \begin{cases}
 \frac{q}{q-1} \int_{a^{q-1}}^1 v^{\frac{1}{q-1}}dv \le \frac{q}{q-1} (1-a^{q-1}) & a<1 \\
  \frac{q}{q-1} \int_1^{a^{q-1}}  v^{\frac{1}{q-1}}dv \le \frac{q}{q-1} (a^{q-1}-1)a & a>1
 \end{cases}.
\end{align}  
Similarly, 
\begin{align}\label{e:aq2}
|a^q-1| = \begin{cases}
 q \int_{a}^1 v^{q-1}dv \le q (1-a) & a<1 \\
 q \int_1^{a}  v^{q-1}dv \le q(a-1)a^{q-1} & a>1
 \end{cases}.
\end{align}
The desired inequality follows by taking the product of \eqref{e:aq1}-\eqref{e:aq2}. 
\end{proof}

Via truncation we arrive at some kind of logarithmic Sobolev inequality for the power function of non negative, non increasing  Poisson functionals. 

\begin{corollary}\label{c:G^q}
{Let $G\geq 0$ be $\sigma(\eta)$-measurable, and such that $DG\le 0$}. For any $q>1$ {such that $\EE G^q <\infty$}, 
\begin{align*}
\mathrm{Ent}(G^q)\le \frac{q^2}{q-1}\EE[\Gamma(G^{q-1},G)].
\end{align*}
\end{corollary}
\begin{proof} Assume first that $G>0$ a.s.. {By \eqref{e:logsob^2}} and Lemma \ref{l:tech}, we have
\begin{align*}
\mathrm{Ent}(G^q)\le \EE\int \frac{(D_x G^q)^2}{G^q} \la(dx) \le \frac{q^2}{q-1} \EE\int (D_xG^{q-1})(D_xG)\la(dx).
\end{align*}
Since the last integral can be written in terms of the carr\'e du champs operator { (see the discussion at the end of Section \ref{s:calculsto}, as well as \cite{DP}) }, we thus have 
$$\EE\int (D_xG^{q-1})(D_xG)\la(dx) =\EE[\Gamma(G^{q-1},G)],$$
as desired. {To deal with the general case}, {let $G_\ep :=G\vee\ep$}, with $\ep>0$ and $x\vee y := \max\, \{ x,y\} $ for all $x,y\in\RR$. Then $G_\ep>0$ and $DG_\ep\le 0$, thus
\begin{align*}
\mathrm{Ent}(G_\ep^q)\le \frac{q^2}{q-1}\EE[\Gamma(G_\ep^{q-1},G_\ep)].
\end{align*}
Note that
\begin{multline*}
\mathrm{Ent}(G_\ep^q)
 = q\EE[(G\vee\ep)^q\log(G\vee\ep)\ind{G\ge \blue{2} }] \\ + q\EE[(G\vee\ep)^q\log (G\vee\ep)\ind{G< \blue{2} }] - \EE[(G\vee\ep)^q]\log\EE[(G\vee\ep)^q].
\end{multline*}
 Now, since $G^q$ is integrable, one obtains immediately that $ \EE[(G\vee\ep)^q] \rightarrow \EE G^q$, as $\ep \downarrow 0$. On the other hand, $q\EE[(G\vee\ep)^q\log(G\vee\ep)\ind{G< \blue{2} }] \to q\EE[G^q\log(G)\ind{G< \blue{2} }]$ by dominated convergence, and $q\EE[G^q\log(G)\ind{G\ge \blue{2} }] \leq \liminf_{\ep\downarrow 0} q\EE[(G\vee\ep)^q\log(G\vee\ep)\ind{G\ge \blue{2} }]$, by Fatou's lemma. From this, we infer that $ \mathrm{Ent}(G^q)\leq \liminf_{\ep\downarrow 0} \mathrm{Ent}(G_\ep^q)$. \red{\blue{To conclude}, we write explicitly 
\begin{align*}
&\EE[\Gamma(G_\ep^{q-1},G_\ep)] \\
&= \EE \left[\int (G(\eta+\de_x)\vee\ep)^{q-1} - (G(\eta)\vee\ep)^{q-1})(G(\eta+\de_x)\vee\ep -G(\eta)\vee\ep) \la(dx)\right],
\end{align*}
\blue{and observe that,} for $a\ge b\ge 0$ and $\ga, \ep>0$, \blue{one has} $(a\vee\ep)^{\ga} - (b\vee\ep)^{\ga}\le (a^\ga - b^{\ga})$, yielding $$\EE[\Gamma(G_\ep^{q-1},G_\ep)]\le \EE[\Gamma(G^{q-1},G)].$$ The desired result follows.}
\end{proof}

%

\begin{remark}
This form of logarithmic Sobolev inequality is nearly optimal when $G=e^{-a\eta(B)}$ with $\la(B)=\ga\in(0,\infty)$, and $a$ sufficiently small. For such a $G$, we have
\begin{align*}
D_x G &= G (e^{-a}-1)1_B(x) \\
D_x G^{q-1} &= G^{q-1} (e^{-a(q-1)}-1)1_B(x). 
\end{align*}
We thus have
\begin{align*}
\EE\int (D_xG^{q-1})(D_xG)\la(dx) = \EE[G^q](1-e^{-a})(1-e^{-(q-1)a})\ga
\end{align*}
On the other hand, standard calculation for Poisson variables gives 
\begin{align*}
\EE[G^q] &= \exp(\ga(e^{-qa}-1)) \\
\EE[G^q\log(G^q)] &= -aq\ga e^{-aq} \exp((e^{-aq}-1)\ga)
\end{align*}
so that 
\begin{align*}
\mathrm{Ent}(G^q)=\ga \exp(\ga(e^{-qa}-1))(1-aqe^{-aq}-e^{-aq}).
\end{align*}
The corollary establishes thus
\begin{align*}
1-aqe^{-aq}-e^{-aq}\le \frac{q^2}{q-1}(1-e^{-a})(1-e^{-(q-1)a}).
\end{align*}
Letting $q\downarrow 1$, we have
\begin{align*}
1-ae^{-a}-e^{-a} \le a(1-e^{-a})
\end{align*}
where the equality is achieved when $a=0$. 
\end{remark}

\subsection{Proof of Theorem \ref{t:hyper}} 

We adapt the classic proof in \cite[p.247]{BGL}. Let \red{$F\ge 0$} and $q>1$. \red{By \blue{truncation and a} monotone convergence argument, we can assume that $F$ is bounded. In particular, $\EE[F^q]<\infty$.}  Define $\Lambda:=\Lambda(t,q)=E[(P_tF)^q]$. We have
\begin{align}\label{e:dt}
\partial_t \Lambda(t,q) &= q \EE[(P_tF)^{q-1} LP_t F] = -q\EE[\Gamma((P_tF)^{q-1},P_tF)], \\
\partial_q \Lambda(t,q) &= \EE[(P_tF)^q\log P_tF] = \frac{1}{q}(\mathrm{Ent}((P_tF)^q)+ \Lambda\log\Lambda).\label{e:dq}
\end{align}
Now let $q(t)=1+ (p-1)e^t$ for $p>1$ and $t>0$. Thus $q'(t)=q(t)-1$.  Consider $H(t)=\frac{1}{q(t)}\log \Lambda(t,q(t))$. Since $H(0)=\log \norm{F}_p$, it remains to prove that 
\begin{align*}
H'(t)=\frac{(\partial_t \Lambda + q'(t)\partial_q\Lambda )q(t)/\Lambda - q'(t)\log \Lambda}{q(t)^2} \le 0
\end{align*}
for $t>0$. Using \eqref{e:dt}-\eqref{e:dq}, this amounts to
\begin{align*}
\mathrm{Ent}((P_tF)^{q(t)}) \le \frac{q(t)^2}{q(t)-1}E[(\Gamma(P_tF)^{q-1},P_tF)]
\end{align*}
which is precisely Corollary \ref{c:G^q} with $G=P_tF$ and $q=q(t)$. The proof is complete.  \qed

%

\subsection{Proof of Theorem \ref{t:tal}}

%

We adapt the arguments from \cite{CEL}.  By $\Var(F) = \EE[(P_0F)^2] - \EE[(P_\infty F)^2]$ and interpolation, we have
\begin{align*}
\Var(F) = -\int_0^\infty \partial_t \EE[(P_t F)^2] dt = - \int_0^\infty \EE[P_t F LP_t F] dt = \int_0^\infty \EE[\Gamma(P_tF,P_tF)] dt.
\end{align*} 
By the commutation relation (\ref{e:commutativity}), the above integral becomes
\begin{equation}\label{az}
\int_0^\infty dt \,e^{-2t} \int \lambda(dx)\,\EE[(P_t D_xF)^2].
\end{equation}
Applying the hypercontractivity (Theorem \ref{t:hyper}) with $p=1+e^{-t}$ to $D_xF$ if $DF\ge 0$, to $-D_xF$ otherwise, we obtain that (\ref{az}) is bounded from above by
\begin{align*}
\int_0^\infty dt\,e^{-2t} \int\la(dx)\, \EE[|D_xF|^{1+e^{-t}}]^{\frac{2}{1+e^{-t}}}.
\end{align*}
After the change of variable $v=1+e^{-t}$, we obtain 
\begin{eqnarray*}
\Var(F) \le  \int \la(dx) \int_1^2  dv\,(v-1)[\EE|D_x F|^v]^{2\over v} dv &\le& \int\la(dx)\int_1^2 dv\, [\EE|D_x F|^v]^{2\over v}\\
&=& \int\la(dx)\int_1^2 dv\, [\EE|D_x F|^{2-v+2(v-1)}]^{2\over v}.
\end{eqnarray*}
By H\"older's inequality with $1/p=2-v$ and $1/q=v-1$, we have
\begin{align*}
\EE|D_x F|^{2-v+2(v-1)} \le \norm{D_x F}_1^{2-v}\norm{D_xF}_2^{2(v-1)}.
\end{align*} 
Setting $b= \norm{D_xF}_1 /\norm{D_xF}_2$, this implies 
\begin{eqnarray*}
\Var(F)&\leq& \int\la(dx)\EE[(D_xF)^2] \int_1^2 b^{\frac4v-2}dv \\
&\leq&\frac14\int\la(dx)\EE[(D_xF)^2] \int_0^2 b^{u}du
\leq\frac12 \int \frac{\EE[(D_xF)^2]}{1+\log (1/b)}\la(dx),
\end{eqnarray*}
see \cite[p.~9]{CEL} for the last inequality.
The desired conclusion follows.\qed

\subsection{Proof of Theorem \ref{t:L^1}} 




The Poincar\'e inequality implies that following variance bound: 
\begin{align*}
\Var(F) \le \frac{1}{1-e^{-1}} \left(\EE[F^2] - \EE[(P_1 F)^2]\right),
\end{align*}
see \cite[Equation (10)]{CEL}.  Applying the semigroup interpolation argument leading to Theorem \ref{t:tal} and the hypercontractivity, we have  
\begin{align*}
\Var(F) &\le  \frac{1}{1-e^{-1}}\int \la(dx) \int_0^1dt\, e^{-2t}  \EE[(P_t D_xF)^2]\\
&\le \frac{1}{1-e^{-1}} \int  \la(dx) \int_0^1dt\, e^{-2t}  \EE[|D_xF|^{1+e^{-t}}]^{\frac{2}{1+e^{-t}}} dt.
\end{align*}
An application of the trivial bound $|D_xF|\le 2\norm{F}_\infty$ yields, with $\alpha(F)=1$ if $2\norm{F}_\infty>1$ and $\alpha(F)=2/(e+1)$ otherwise,
\begin{eqnarray*}
\Var(F)&\le& \frac{(2\norm{F}_\infty)^{\alpha(F)}}{1-e^{-1}} \int \la(dx)\int_0^1dt\, e^{-2t}  \EE[|D_xF|]^{\frac{2}{1+e^{-t}}} \\
&\le& \frac{(2\norm{F}_\infty)^{\alpha(F)}}{1-e^{-1}} 
\int \la(dx)\int_{2/(1+e^{-1})}^1du\, \EE[|D_xF|]^{u}\frac{2}{u^2} \\
&\le& 11 \, (2\norm{F}_\infty)^{\alpha(F)}
\int \la(dx)\int_{0}^1  \EE[|D_xF|]^{u} du\\
&\le& 11 \, (2\norm{F}_\infty)^{\alpha(F)}
\int \la(dx)\times \left\{
\begin{array}{cl}
2\int \frac{\la(dx)}{1+\log(1/\EE[|D_xF|])}&\mbox{if $\EE[|D_xF|]\leq 1$}\\
\EE[|D_xF|]&\mbox{if $\EE[|D_xF|]\geq 1$}\\
\end{array}
\right..
\end{eqnarray*}
\qed

\section{An $L^p$-gradient estimate}\label{s:grad}

Our last result may be compared with the classical gradient estimate in the Gaussian space: for $|f|\le 1$ and $0< t\le 1$,
\begin{align*}
\left|\nabla (P^{\ga_k}_t f)\right| \le 
\frac{e^{-t}}
{
\sqrt{1-e^{-2t}}
},
\end{align*}
see \cite[Prop.~5.1.5]{NN} (see also \cite[Equation~(25)]{CEL}).

\begin{theorem}\label{t:grad_integ}
Let $F\in L^p(\mathbb{P})$ be such that $\norm{DF}_{L^p(\Omega;L^2(\la))}<\infty$ for some  $p\in [2,\infty]$. Then for $t>0$,  
\begin{align*}
\norm{D P_t F}_{L^p(\Omega; L^2(\la))}  \le \frac{e^{-t}}{\sqrt{1-e^{-t}}} \norm{F}_p. 
\end{align*}
\end{theorem}
\begin{proof}
i) Let $2\le p<\infty$. By duality, it suffices to show that
\begin{align*}
\left| \EE \int h(x)G D_x P_t F  \la(dx) \right|  \le  \frac{e^{-t}}{\sqrt{1-e^{-t}}} \norm{F}_p \norm{hG}_{L^{p'}(\Omega;L^2(\la))} 
\end{align*}
for any $h\in L^2(\la)$ and $G\in L^{p'}$, where $\frac{1}{p}+ \frac{1}{p'}=1$. By \eqref{e:mehler}, we have
\begin{align*}
P_t D_x F = \EE'[D_x F(\eta_{e^{-t}}+ \eta'_{1-e^{-t}})]
\end{align*}
where $\EE'$ is the expectation with respect to the distribution of $\eta'_{1-e^{-t}}$. Thus,
with $\lambda_t=(1-e^{-t})\lambda$ denoting the intensity measure of $\eta'_{1-e^{-t}}$ and by using the commutation relation (\ref{e:commutativity}) in the first equality,
\begin{align}\label{e:p10}
\int h(x) D_x P_t F\la(dx) &= e^{-t} \int h(x) P_t D_x F\la(dx)\nonumber\\
& = \frac{e^{-t}}{1-e^{-t}} \EE'\int D_x F(\eta_{e^{-t}}+\eta'_{1-e^{-t}}) h(x)\la_t(dx)=:\frac{e^{-t}}{1-e^{-t}} Y_t
\end{align}
We now proceed with Mecke's formula (\ref{mecke}) to arrive at
\begin{align}\label{e:p10a}
Y_t &= \EE' \int F(\eta_{e^{-t}}+ \eta'_{1-e^{-t}}+\de_x) h(x) \la_t(dx)  - \EE' \int F(\eta_{e^{-t}}+ \eta'_{1-e^{-t}}) h(x)\la_t(dx)\nonumber \\
&= \EE' \left[ F(\eta_{e^{-t}}+ \eta'_{1-e^{-t}})  \int h(x)(\eta'_{1-e^{-t}}(dx) - \la_t(dx))  \right].
\end{align}
By the Cauchy-Schwarz inequality and \eqref{e:mehler},
\begin{eqnarray*}
|Y_t|\le \EE'[  F(\eta_{e^{-t}}+ \eta'_{1-e^{-t}})^2 ]^{1\over 2} \EE'[I'_1(h)^2]^{1\over 2} &=& [P_t (F^2)]^{1\over 2} \norm{h}_{L^2(\la_t)}\\ & =& \sqrt{1-e^{-t}} [P_t(F^2)]^{1\over 2}\norm{h}_{L^2(\la)},
\end{eqnarray*}
where $I'_1(h)$ is the compensated Poisson integral of order $1$ of $h$ with respect to $\eta'_{1-e^{-t}}$.  These estimates, together with the use of H\"older's inequality and of the fact that $P_t$ is a contraction on $L^q$, $1\le q\le \infty$, leads to 
\begin{align*}
\left| \EE \int h(x)G D_x P_t F  \la(dx) \right| &\le \frac{e^{-t}}{\sqrt{1-e^{-t}}}\norm{h}_{L^2(\la)} \EE[ |[P_t (F^2)]^{1\over 2}G| ] \\
&\le  \frac{e^{-t}}{\sqrt{1-e^{-t}}} \norm{h}_{L^2(\la)} \norm{G}_{p'} \norm{P_t (F^2)}_{p\over 2}^{1\over 2}\\
&\le \frac{e^{-t}}{\sqrt{1-e^{-t}}}\norm{h}_{L^2(\la)} \norm{G}_{p'} \norm{F^2}_{p\over 2}^{1\over 2} \\
&=  \frac{e^{-t}}{\sqrt{1-e^{-t}}}\norm{F}_{p} \norm{hG}_{L^{p'}(\Omega;L^2(\la))},
\end{align*}
as desired.  

ii) Suppose $p=\infty$, then $F$ is a bounded functional. The proof goes along similar lines as item i). Combining \eqref{e:p10}-\eqref{e:p10a}, the boundedness of $F$ and the Cauchy-Schwarz inequality, we have
\begin{align*}
|\langle D_x P_t F, h\rangle_{L^2(\la)}|= \frac{e^{-t}}{1-e^{-t}} |Y_t|  \le  \frac{e^{-t}}{\sqrt{1-e^{-t}}} \norm{F}_\infty \norm{h}_{L^2(\la)} \quad \mbox{ a.s.}
\end{align*}
for all $h\in L^2(\la)$. It follows that
\begin{align*}
\norm{D_x P_t F}_{L^2(\la)} = \sup_{h: \norm{h}_{L^2(\la)}\le 1} \langle D_x P_t F, h\rangle_{L^2(\la)} \le \frac{e^{-t}}{\sqrt{1-e^{-t}}}\norm{F}_\infty \quad \mbox{ a.s.}
\end{align*}
ending the proof. 
\end{proof}

\bibliographystyle{plain}

\begin{thebibliography}{123}

\bibitem{LS2000} C. An\'e; S. Blach\`ere, D. Chafa\"i, P. Foug\`eres, I. Gentil, F. Malrieu, C. Roberto, G. Scheffer (2000). {\it Sur les inégalités de Sobolev logarithmiques}. Panoramas et Synth\`eses, SMF.

\bibitem{AL} C. An\'e and M. Ledoux (2000). On logarithmic Sobolev inequalities for continuous time random walks on graphs. {\it Probab. Theory Related Fields}, {\bf 116}(4), 573--602.

\bibitem{B} S. Bachmann (2016). Concentration for Poisson functionals: component counts in random geometric graphs. 
{\it Stochastic Processes and their Applications}, {\bf 126}(2016), 1306-1330

\bibitem{BP}S. Bachmann and G. Peccati. Concentration bounds for geometric Poisson functionals: logarithmic Sobolev inequalities revisited. {\it Electron. J. Probab.} {\bf 21}, Paper No. 6, 44 pp.

\bibitem{BR} S. Bachmann and M. Reitzner (2018). Concentration for Poisson U-Statistics: Subgraph Counts in Random Geometric Graphs. {\it Stochastic Process. Appl.}, {\bf 128}, 3327-3352 (2018).

\bibitem{BGL}D. Bakry, I. Gentil and M. Ledoux (2014). {\it Analysis and geometry of Markov diffusion operators}. Grundlehren der Mathematischen Wissenschaften [Fundamental Principles of Mathematical Sciences], {\bf 348}. Springer-Verlag.

\bibitem{BL}S.G. Bobkov and M. Ledoux (1998). On modified logarithmic Sobolev inequalities for Bernoulli and Poisson measures. {\it J. Funct. Anal.}, {\bf 156}(2), 347--365. 

\bibitem{BT}S. G. Bobkov and P. Tetali (2006). Modified logarithmic Sobolev inequalities in discrete settings. {\it J. Theoret. Probab.} {\bf 19 }(2), 289--336.

\bibitem{BHP} J.-Ch. Breton, Ch. Houdr\'e and N. Privault (2007). Dimension free and infinite variance tail estimates on Poisson space. {\it Acta Appl. Math.}, {\bf 95}(3), 151-203

\bibitem{Chafai} D. Chafa\"i (2004). Entropies, convexity, and functional inequalities, On $\Phi$-entropies and $\Phi$-Sobolev inequalities. {\it J. Math. Kyoto Univ.}, {\bf 44}( 2), 325-363.

\bibitem{C} S. Chatterjee (2014). {\it Superconcentration and related topics}. Springer-Verlag.

\bibitem{CEL}D. Cordero-Erausquin and M. Ledoux (2012). Hypercontractive measures, Talagrand's inequality, and influences. {\it Geometric aspects of functional analysis}, 169--189, Lecture Notes in Math. 2050, Springer-Verlag.


\bibitem{DP}Ch. Döbler and G. Peccati (2018). The fourth moment theorem on the Poisson space. {\it Ann. Probab.} {\bf 46 }(4), 1878--1916.


\bibitem{GLO} F. Gieringer and G. Last (2018). Concentration inequalities for measures of a Boolean model. 
{\it ALEA} {\bf 15}, 151–166.

\bibitem{HP} Ch. Houdr\'e and N. Privault (2002). Concentration and deviation inequalities in infinite dimensions via covariance representations. {\it Bernoulli}, {\bf 8}(6), 697-720.

\bibitem{L} G. Last (2016). 6] G. Last (2016). Stochastic analysis for Poisson processes. In:  G. Peccati and M. Reitzner (Editors) (2016). {\it Stochastic Analysis for Poisson Point Processes.
Malliavin Calculus, Wiener–It\^o Chaos Expansions and Stochastic Geometry}. Springer \& Bocconi University Press, Bocconi \& Springer Series Vol. 7


\bibitem{LP}G. Last and M. Penrose (2018). {\it Lectures on the Poisson process}. Institute of Mathematical Statistics Textbooks, {\bf 7}. Cambridge University Press, Cambridge.

\bibitem{L2000} M. Ledoux (2000). The geometry of Markov diffusion generators. {\it Annales de la facult\'e des sciences de Toulouse}, {\bf 9}(2), 305-366.


\bibitem{Ledoux} M. Ledoux (2001). {\it The concentration of measure phenomenon}. Mathematical Surveys and Monographs, {\bf 89}. American Mathematical Society.

\bibitem{NPbook} I. Nourdin and G. Peccati (2012). {\it Normal approximations using Malliavin calculus. From Stein's method to universality}. Cambridge University Press.

\bibitem{NN} D. Nualart and E. Nualart (2018). {\it Introduction to Malliavin Calculus}.  IMS Textbooks, Cambridge  University Press.

\bibitem{PR} G. Peccati and M. Reitzner (Editors) (2016). {\it Stochastic Analysis for Poisson Point Processes.
Malliavin Calculus, Wiener–It\^o Chaos Expansions and Stochastic Geometry}. Springer \& Bocconi University Press, Bocconi \& Springer Series Vol. 7

\bibitem{R} M. Reizner (2013). Poisson point processes: Large deviation inequalities for the convex distance. {\it Electron. Commun. Probab. }, {\bf 18}(96), 1-7.

\bibitem{Surgailis} D. Surgailis (1984). On multiple Poisson stochastic integrals and associated Markov semigroups. {\it Probab. Math. Statistics}, {\bf 3}(2), 217-239.

\bibitem{T} M. Talagrand (1994). On Russo's approximate zero-one law. {\it Ann. Probab.} {\bf 22}(3), 1576--1587. 

\bibitem{W}
L. Wu (2000). A new modified logarithmic Sobolev inequality for Poisson point processes and several applications. {\it Probab. Theory Related Fields}, {\bf 118}(3), 427--438.






\end{thebibliography}

\end{document}